\documentclass[10pt]{amsart}
\usepackage{amsmath, amscd, amsthm} 
\usepackage{amssymb, amsfonts} 
\usepackage[all]{xy} 

\newcommand{\C}{\mathbb{C}} 
\newcommand{\A}{\mathbb{A}}
\renewcommand{\P}{\mathbb{P}}
\newcommand{\Z}{\mathbb{Z}}
\newcommand{\R}{\mathbb{R}}

\newcommand{\iso}{\cong}

\newcommand{\wkeq}{\simeq}

\newcommand{\Hom}{\mathrm{Hom}}
\newcommand{\mcal}[1]{\mathcal{#1}}

\DeclareMathOperator*{\colim}{\mathrm{colim}}

\DeclareMathOperator{\hofib}{\mathrm{hofib}}

\DeclareMathOperator{\sing}{\mathrm{Sing}}

\DeclareMathOperator{\spec}{\mathrm{Spec}}

\DeclareMathOperator{\coker}{\mathrm{Coker}}
\DeclareMathOperator{\codim}{\mathrm{codim}}

\numberwithin{equation}{section} 

\theoremstyle{plain}
\newtheorem{theorem}[equation]{Theorem}
\newtheorem{proposition}[equation]{Proposition}

\newtheorem{corollary}[equation]{Corollary}
\newtheorem{conjecture}[equation]{Conjecture}

\theoremstyle{definition}
\newtheorem{definition}[equation]{Definition}
\newtheorem{example}[equation]{Example}

\theoremstyle{remark}
\newtheorem{remark}[equation]{Remark}

\begin{document}
\title{ Remarks on filtrations of the homology of real varieties}
\author{Jeremiah Heller}
\email{heller@math.uni-wuppertal.de}
\author{Mircea Voineagu}
\email{mircea.voineagu@ipmu.jp}

\begin{abstract}
 We demonstrate that a conjecture of Teh which relates the niveau filtration on Borel-Moore homology of real varieties and the images of generalized cycle maps from reduced Lawson homology is false. We show that the niveau filtration on reduced Lawson homology is trivial and construct an explicit class of examples for which Teh's conjecture fails by generalizing a result of Sch\"ulting. 
 We compare various cycle maps and in particular we show that the Borel-Haeflinger cycle map naturally factors through the reduced Lawson homology cycle map.
\end{abstract}
\maketitle
\tableofcontents
\section{Introduction}

Let $X$ be a quasi-projective real variety. In \cite{Teh:real} the reduced Lawson homology groups $RL_{q}H_{n}(X)$ are introduced as homotopy groups of certain spaces of ``reduced'' algebraic cycles. For $q=0$ we have $RL_{0}H_{n}(X) = H_{n}(X(\R);\Z/2)$, where $H_{n}(X(\R);\Z/2)$ is the Borel-Moore homology. At the other extreme we have that $RL_{n}H_{n}(X)$ is a quotient of the Chow group $CH_{n}(X)$. There are  generalized cycle maps
$$
cyc_{q,n}:RL_{q}H_{n}(X)\to H_{n}(X(\R);\Z/2)
$$
and the images of these cycle maps form a filtration of the homology
$$
Im(cyc_{n,n}) \subseteq Im(cyc_{n-1,n}) \subseteq \cdots \subseteq Im(cyc_{0,n})=H_{n}(X(\R);\Z/2).
$$
The first step of this filtration is the image of the Borel-Haeflinger cycle map $Im(cyc_{n,n}) =H_{n}(X(\R);\Z/2)_{alg}$ . 

The construction of the reduced Lawson homology is based on Friedlander's construction of Lawson homology groups for complex varieties. Friedlander-Mazur \cite{FM:filt} have conjectured a relationship between the filtration on singular homology of the space of complex points given by images of the generalized cycle map and the niveau filtration. Teh makes an analogous conjecture for the reduced Lawson groups.
\begin{conjecture}[{\cite[Conjecture 7.9]{Teh:real}}]\label{tco}
Let $X$ be a smooth projective real variety. Then   $Im(cyc_{q,n}) =N _{2n-q}H _n(X(\R);\Z/2)$ for any $0\leq q\leq n$.
\end{conjecture}  
Here $N_pH _n(X(\R);\Z/2)$ is the niveau filtration which is the sum over all images  
$$
Im\left(H_n(V(\R);\Z/2)\rightarrow H_n(X(\R);\Z/2)\right)
$$ 
such that $\dim V \leq p$. 

In the complex case, Friedlander-Mazur's conjecture is a very difficult and interesting question. It is known to be true with arbitrary finite coefficients (in particular with $\Z/2$-coefficients), as a corollary of Beilinson-Lichtenbaum conjecture. With integral coefficients, it is known that Suslin's conjecture for a smooth quasi-projective variety $X$ implies Friedlander-Mazur's conjecture for $X$. 

Surprisingly, the real case is totally different. Our main result says that the analog of this conjecture over the field of real numbers is false. 
\begin{theorem}
 Conjecture \ref{tco} is false.
\end{theorem}

To see the failure of this conjecture we first observe that the niveau filtration on reduced Lawson homology is uninteresting (the case of Borel-Moore homology is $q=0$). Specifically we have that
$$
N_{j}RL_{q}H_{n}(X) = 
\begin{cases}
RL_{q}H_{n}(X) & j\geq n \\
0 & \textrm{else} .
\end{cases}
$$
This is a consequence of Corollary \ref{dge} which asserts that the coniveau spectral sequence for the reduced morphic cohomology collapses. Conjecture \ref{tco} is therefore equivalent to the surjectivity of the generalized cycle maps. It is known that  in general $H_{n}(X(\R);\Z/2) \neq H_{n}(X(\R);\Z/2)_{alg}$ although it is difficult to find explicit examples. In Example \ref{example} and Proposition \ref{genex} we give an explicit class of such examples.  Our examples are based on a decomposition given in Theorem \ref{decompthrm} of the reduced cycle spaces of a blow-up with smooth center. This decomposition is a generalization to reduced cycle spaces of the main result of \cite{S:reelle}.

As another application of thee decomposition we have the following theorem, which is also different from the complex case where the group of divisors modulo algebraic equivalence of an irreducible complex variety always injects into the corresponding homology group. Theorem \ref{intt} gives  examples of thin divisors that are non-trivial in the reduced Lawson group (see Remark \ref{thin}). 
\begin{theorem} \label{intt}
There exists  a smooth real variety $X\stackrel{birational}\sim \mathbb{P}^N _\mathbb{R}$ such that the cycle map on divisors $RL^1H^1(X)\rightarrow H^1(X(\R),\Z/2)$ is not injective.
\end{theorem} 

The collapse of the coniveau spectral sequence for reduced morphic cohomology is a consequence of local vanishing of motivic cohomology in degrees larger than the weight together with the vanishing theorem proved in \cite{HV:VT}. For the purposes of seeing that Conjecture \ref{tco} is false one does not need the full strength of the collapsing, it suffices to use only the local vanishing of motivic cohomology. However, an interesting consequence of the collapse of this spectral sequence is that we can compute reduced morphic cohomology as the sheaf cohomology 
$$
H^{n}_{Zar}(X;\mcal{RL}^{q}\mcal{H}^{0})= RL^{q}H^{n}(X).
$$ 
As a consequence, we identify a family of birational invariants given by 
$$
H^{0}_{Zar}(X;\mcal{RL}^{q}\mcal{H}^{0})= RL^{q}H^{0}(X),
$$
for any $q\geq 0$. In case $q=\dim(X)$ we obtain that the number $s$ of connected components of $X(\mathbb{R})$ is a birational invariant of an algebraic nature (i.e. $RL^{\dim(X)}H^{0}(X)=(\Z/2)^s$).
 The purely algebraic nature of $s$ forms part of the main result of \cite{CTP:real}, where they use \'etale cohomology. The relation between reduced morphic cohomology and \'etale cohomology is discussed in the final section. As an application of these birational invariants  we compute reduced Lawson homology of a real rational surface in Corollary \ref{comp}. 
 
In Section \ref{cyclemaps} we discuss various cycle maps. We show that there is basically one cycle map from the mod-$2$ motivic cohomology to mod-$2$ singular cohomology of the space of real points.  As a consequence we see in Theorem \ref{bh} that the Borel-Haeflinger cycle map factors through this cycle map.

\section{Preliminaries}
Let $Y$ be a projective complex variety and $\mcal{C}_q(Y)$ the Chow variety of effective $q$-cycles on $Y$. Write $\mcal{Z}_q(Y) = (\mcal{C}_q(Y)(\C))^{+}$ for the group completion of this monoid. The group completion is done algebraically and $\mcal{Z}_q(Y)$ is given the quotient topology. It turns out that this naive group completion is actually a homotopy group completion \cite{FG:cyc}, \cite{LF:qproj}. When $U$ is quasi-projective with projectivization $U\subseteq \overline{U}$ then define $\mcal{Z}_{q}(U) = \mcal{Z}_{q}(\overline{U})/\mcal{Z}_{q}(U_{\infty})$ where $U_{\infty} = \overline{U}\backslash U$ (and the quotient is a group quotient). This definition is independent of choice of projectivization \cite{LF:qproj}, \cite{FG:cyc}.

If $X$ is a real variety then $G=\Z/2$  acts on $\mcal{Z}_{q}(X_{\C})$ via complex conjugation. The space of real cycles on $X$ is defined to be the subgroup $\mcal{Z}_{q}(X_{\C})^{G}$ of cycles invariant under conjugation. Write $\mcal{Z}_{q}(X_{\C})^{av}$ for the subgroup generated by cycles of the form $\alpha+\overline{\alpha}$. The space of reduced cycles on $X$ is defined to be the quotient group
$$
\mcal{R}_{q}(X) = \frac{\mcal{Z}_{q}(X_{\C})^{G}}{\mcal{Z}_{q}(X_{\C})^{av}}.
$$

\begin{definition}[\cite{Teh:real}]
Let $X$ be a quasi-projective real variety. The \textit{reduced Lawson homology} of $X$ is defined by 
$$
RL_{q}H_{q+i}(X) = \pi_{i}\mcal{R}_{q}(X).
$$
\end{definition}

When $q=0$ we have that $\mcal{R}_{0}(X) = \mcal{Z}_{0}(X(\R))/2$ so by the Dold-Thom theorem $RL_{0}H_{i}(X) = H_{i}(X(\R);\Z/2)$ is the Borel-Moore homology of $X(\R)$. In general $RL_{q}H_{q+i}(X)$ are all $\Z/2$-vector spaces. It is not known whether these are finitely-generated vector spaces or not however we do have the following vanishing theorem.
\begin{theorem}[\cite{HV:VT}]
 Let $X$ be a quasi-projective real variety. Then
$$
RL_{q}H_{q+i}(X) = 0
$$
 if $q+i>\dim(X)$.
\end{theorem}

There is also a space $\mcal{R}^{q}(X) = \mcal{Z}^{q}(X_{\C})^{G}/\mcal{Z}^{q}(X_{C})^{av}$ of reduced algebraic cocycles on $X$ when $X$ is normal and projective. We refer to \cite{Teh:real} for the details of its construction. It is convenient to extend this definition to quasi-projective normal varieties, which is done in \cite{HV:VT} although not introduced formally as such an extension. We avoid difficulties with point-set topology by giving the extension as a simplicial abelian group. Define the simplicial abelian group of reduced cocyles on a quasi-projective normal real variety 
$$
\widetilde{\mcal{R}}^{q}(X) = \frac{\sing_{\bullet}{Z}/2^{q}(X_{\C})^{G}}{\sing{Z}_{\bullet}/2^{q}(X_{\C})^{av}}.
$$
 If $X$ is a projective, normal real variety then $\widetilde{R}^{q}(X) \xrightarrow{\wkeq} \sing_{\bullet}\mcal{R}^{q}(X)$ is a homotopy equivalence \cite[Lemma 6.7]{HV:VT}.
\begin{definition}[\cite{Teh:real}]
Let $X$ be a normal quasi-projective variety. Define the reduced morphic cohomology of $X$ by 
$$
RL^{q}H^{q-i}(X) = \pi_{i}\widetilde{\mcal{R}}^{q}(X).
$$
\end{definition}

The reduced Lawson homology and reduced morphic cohomology are related by a Poincare duality. This is proved in \cite[Theorem 6.2]{Teh:real} for smooth projective varieties. We give a quick proof below which applies to both the projective and quasi-projective case. If $T$ is a topological abelian group we write $\widetilde{T} = \sing_{\bullet}T$ for the associated simplicial abelian group. If $M$ is a $G$-module and $\sigma$ is the nontrivial element of $G$ we write $N= 1+\sigma$ and define $M^{av}= Im(N)$. If in addition $M$ is $2$-torsion then we have the two fundamental short exact sequences of abelian groups
$0\to M^{G} \to M \xrightarrow{N} M^{av} \to 0$  and $0\to M^{av} \to M^{G} \to M^{G}/M^{av} \to 0 $.  

\begin{theorem}\label{pd}
Let $X$ be a smooth quasi-projective real variety of dimension $d$. The inclusion
$$
\widetilde{\mcal{R}}^{q}(X) \to \widetilde{\mcal{R}}_{d}(X\times \A^{q})
$$ 
is a homotopy equivalence of simplicial abelian groups. 
Consequently there is a natural isomorphism $RL^{q}H^{n}(X)= RL_{d-q}H_{d-q}(X)$.
\end{theorem}
\begin{proof}
 This follows immediately from consideration of the following comparison diagrams of homotopy fiber sequences of simplicial abelian groups, where the displayed homotopy equivalences follow from \cite[Theorem 5.2]{F:algco} and  \cite[Corollary 4.20]{HV:VT}

 \begin{equation*}
 \xymatrix{
\widetilde{\mcal{Z}}^{q}/2(X_{\C})^{G} \ar[r]\ar[d]^{\wkeq} & \widetilde{\mcal{Z}}^{q}/2(X_{\C}) \ar[d]^{\wkeq}\ar[r]^{N} & \widetilde{\mcal{Z}}^{q}/2(X_{\C})^{av}\ar[d]  \\
\widetilde{\mcal{Z}}_{d}/2(X_{\C}\times_{\C}\A_{\C}^{q}))^{G} \ar[r] & \widetilde{\mcal{Z}}_{d}/2(X_{\C}\times_{\C}\A_{\C}^{q})) \ar[r]^{N} & \widetilde{\mcal{Z}}_{d}/2((X_{\C}\times_{\C}\A_{\C}^{q})^{av}
}
\end{equation*}
and
\begin{equation*}
 \xymatrix{
\widetilde{\mcal{Z}}^{q}/2(X_{\C})^{av} \ar[r]\ar[d] & \widetilde{\mcal{Z}}^{q}/2(X_{\C})^{G} \ar[r]\ar[d]^{\wkeq} & \widetilde{\mcal{R}}^{q}(X) \ar[d]\\
\widetilde{\mcal{Z}}_{d}/2(X_{\C}\times_{\C}\A^{q}_{\C})^{av} \ar[r] & \widetilde{\mcal{Z}}_{d}/2(X_{\C}\times_{\C}\A^{q}_{\C})^{G} \ar[r] & \widetilde{\mcal{R}}_{d}(X\times\A^{q}_{\R}) .
}
\end{equation*}

\end{proof}

The inclusion of algebraic cocycles into topological cocycles gives a generalized cycle map
$$
cyc_{q,n}:RL^{q}H^{n}(X) \to H^{n}(X(\R);\Z/2).
$$
If $X$ is smooth and $q\geq \dim X$ then the cycle map $cyc_{q,n}$ is an isomorphism. For $X$ projective this follows from \cite[Corollary 6.5, Theorem 8.1]{Teh:real} (the results there are stated under the assumption that $X(\R)$ is nonempty and connected but this is unnecessary). The isomorphism for projective varieties implies the isomorphism for quasi-projective varieties (for example by using cohomology with supports and an argument as in \cite[Corollary 4.2]{HV:AHSS}).

There are operations called the $s$-map in both reduced Lawson homology and reduced morphic cohomology $s:RL_{q}H_{n}(X)\to RL_{q-1}H_{n}(X)$ and $s:RL^{q}H^{n}(X)\to RL^{q+1}H^{n}(X)$.
Iterated compositions of $s$-maps give rise to generalized cycle maps in reduced Lawson homology
$$
cyc_{q,n}:RL_{q}H_{n}(X) \xrightarrow{s} RL_{q-1}H_{n}(X) \xrightarrow{s} \cdots \xrightarrow{s} RL_{0}H_{n}(X) = H_{n}(X(\R);\Z/2).
$$ 
The $s$-map in morphic cohomology is compatible with the $s$-map in the sense that $cyc_{q,n}$ agrees with the composition
$$
cyc_{q,n}:RL^{q}H^{n}(X) \xrightarrow{s} RL^{q+1}H^{n}(X) \xrightarrow{cyc_{q,n}} H^{n}(X(\R);\Z/2).
$$

Let $X$ be a normal quasi-projective real variety and $Z\subseteq X$ a closed subvariety. The \textit{reduced morphic cohomology with supports} is defined in the usual way with $RL^qH^{q-i}(X)_Z = \pi_{i}\widetilde{\mcal{R}}^{q}(X)_{Z}$ where $\widetilde{\mcal{R}}^{q}(X)_{Z} = \hofib (\widetilde{\mcal{R}}^{q}(X) \to \widetilde{\mcal{R}}^{q}(X-Z) )$.

\begin{theorem}[Cohomological purity for reduced morphic cohomology]\label{reducedpurity} 
Let $X$ be a smooth, quasi-projective real variety of dimension $d$ and $Z\subset X$  a closed smooth subvariety of codimension $p$. There are homotopy equivalences $\widetilde{\mathcal{R}}^q(X)_Z\wkeq \widetilde{R}^{q-p}(Z)$. These induce natural isomorphisms 
$$
RL^qH^n(X)_Z = RL^{q-p}H^{n-p}(Z).
$$
\end{theorem}
\begin{proof} 
This follows from the localization sequence for reduced Lawson homology 
\cite[Corollary 3.14]{Teh:real} together with Poincare duality between reduced Lawson homology and reduced morphic cohomology  (see \cite[Theorem 6.2]{Teh:real} and Theorem \ref{pd}). 
\end{proof}

Recall that a presheaf $F(-)$ of cochain complexes satisfies \textit{Nisnevich descent} provided that for any smooth $X$, any \'etale map $f:Y\to X$, and open embedding $i:U\subseteq X$ such that $f: Y - f^{-1}(U)\to X - U$ is an isomorphism, we have
a Mayer-Vietoris exact triangle (in the derived category of abelian groups):
\begin{equation*}
 F(X) \to F(Y)\oplus F(U) \to F(f^{-1}(V)) \to F(X)[1].
\end{equation*}

 \begin{corollary} \label{boprops}
The presheaf $\widetilde{\mathcal{R}}^q(-)$ is homotopy invariant theory and satisfies Nisnevich descent.
 \end{corollary}
\begin{proof}
 It is homotopy invariant by \cite[Theorem 5.13]{Teh:real}. Nisnevich descent follows immediately from Theorem \ref{reducedpurity}. 
\end{proof}

\section{Birational invariants and examples}\label{BIE}

We use the following decomposition theorem for spaces of reduced cycles on blow-ups with smooth center in order to obtain a decomposition of the cokernels of the cycle map from reduced Lawson homology. Later we use this decomposition to exhibit spaces whose cycle map has nontrivial cokernel i.e. it is not surjective. Recall that $\mcal{R}_{-q}(X) = \mcal{R}_{0}(X\times\A^{-q})$ for $q\geq 0$.

\begin{theorem}\label{decompthrm}
Let $X$ be a smooth real projective variety and $Z\subset X$ a smooth closed irreducible subvariety of codimension $d>1$. Let $\pi:X^*\rightarrow X$ be the blow up of $X$ by the smooth center $Z$. Then, for any $0\leq q\leq dim(X)$ we have a homotopy equivalence of topological abelian groups
\begin{equation}\label{dec}
   \mcal{R}_q(X^*)\stackrel{h.e.}{\simeq} \mcal{R}_q(X)\oplus \mcal{R}_{q-d+1}(Z)\oplus \mcal{R}_{q-d+2}(Z)\oplus \cdots \oplus \mcal{R}_{q-1}(Z).
\end{equation}
Moreover this decomposition is compatible with $s$-maps.
\end{theorem} 
\begin{proof} 
We follow the ideas used to prove \cite[Theorem 2.5]{Voin:2} and the main theorem of \cite{S:reelle}. We work in $\mathcal{H}^{-1}AbTop$, the category of topological abelian groups with a CW-structure with inverted homotopy equivalences.

 Recall that $\pi^{-1}(Z)\rightarrow Z$ is the projective bundle $p:\mathbb{P}(N _ZX)\rightarrow Z$  of dimension $d-1$. The decomposition in the statement of the theorem is given as follows. The first component of the map  is $\pi_{*}$ (notice that $\pi _*\circ \pi^*=id$). The other maps are given via compositions 
$$
\phi _l: \mcal{R}_{q-d+1+l}(Z)\xrightarrow{p^*} \mcal{R}_{q+l}(X^*)\xrightarrow{-\cap c _1(O(1))^l} \mcal{R}_q(X^*),
$$ 
 where $p^*: \mcal{R}_q(Z)\rightarrow \mcal{R}_{q+d-1}(\mathbb{P}(N _ZX))\stackrel{i _*}{\rightarrow} \mcal{R}_{q+d-1}(X^*)$. 

Using the Mayer-Vietoris sequence (see Corollary \ref{boprops}), we have that 
$$
\mcal{R}_k(X^*)\stackrel{h.e.}{\rightarrow} \mcal{R}_k(X)\oplus Ker(p _*),
$$ 
with $p _*:\mcal{R} _k(N _Z(X))\rightarrow \mcal{R}_k(Z)$ and then identify in $\mathcal{H}^{-1}AbTop$
\begin{equation}\label{khe}
Ker(p _*)\stackrel{h.e.}{\simeq}\mcal{R}_{q-d+1}(Z)\oplus \mcal{R}_{q-d+2}(Z)\oplus\cdots\oplus \mcal{R}_{q-1}(Z).
\end{equation}
Using the properties of Segre classes $s _l(N _ZX)\cap -:\mcal{R}_k(Z)\rightarrow \mcal{R}_{k+d-1-l}(Z)$ i.e. $s _l(N _Z(X))\cap -=0$ if $l<0$ and $s _0(N _Z(X))\cap -=id$, one can prove  the projective bundle formula for the reduced cycle groups in the usual way, see \cite{Teh:real}. Observe that the projective bundle formula works also for negative indexes. We  obtain $$\mcal{R}_k(N _Z(X))=\oplus _{0\leq l\leq d-1}\mcal{R}_{k-d+1+l}(Z).$$
Now one can conclude the homotopy equivalence (\ref{khe}).

The $s$-maps are compatible with all of the maps involved in the decomposition \ref{dec} (see \cite{Teh:real}) and therefore the decomposition of the theorem is preserved by the $s$-maps.    

\end{proof}

The generalized cycle maps $cyc_{q,n}: RL_qH _n(X)\rightarrow H _n(X(\mathbb{R}),\mathbb{Z}/2)$,
are defined as a composite of $s$-maps together with the Dold-Thom isomorphism. Write
$$
T _{q,n}:=\coker(cyc_{q,n}:RL_qH _n(X)\rightarrow H _n(X(\mathbb{R}),\mathbb{Z}/2))
$$ 
and

$$
K _{q,n}(X)=\ker(cyc _{q,n}: RL _qH _n(X)\rightarrow H _n(X(\mathbb{R}),\mathbb{Z}/2)).
$$

Notice that $T _{q,n}=0 = K_{q,n}$, for $q\leq 0$.

\begin{corollary}\label{bma}
Let $\pi:X^{*}\to X$ and $Z$ be as in the above theorem. Then
$$
T _{q,n}(X^*)=T _{q,n}(X)\oplus T _{q-1,n-1}(Z)\oplus...\oplus T _{q-d+1,n-d+1}(Z),
$$ 
and
\begin{equation} 
K _{q,n}(X^*)=K _{q,n}(X)\oplus K _{q-1,n-1}(Z)\oplus...\oplus K _{q-d+1,n-d+1}(Z).
\end{equation}
\end{corollary}
\begin{proof}
In the case $k=0$ the decomposition (\ref{dec}) gives  
\begin{equation*}
  \mcal{R}_0(X^*(\mathbb{R}))\stackrel{h.e.}{\simeq} \mcal{R}_0(X(\mathbb{R}))\oplus \mcal{R}_{-d+1}(Z(\mathbb{R}))\oplus \mcal{R}_{-d+2}(Z(\mathbb{R}))\oplus \cdots \oplus \mcal{R}_{-1}(Z(\mathbb{R})),
\end{equation*}
   therefore producing the decomposition of Borel-Moore homology
\begin{equation*}
   H _k(X^*(\mathbb{R}),\mathbb{Z}/2)=H _k(X(\mathbb{R}),\mathbb{Z}/2)\oplus H _{k-d-1}(Z(\mathbb{R}), \mathbb{Z}/2)\oplus \cdots \oplus H _{k-1}(Z(\mathbb{R}),\mathbb{Z}/2).
\end{equation*}
The $s$-maps respect the decomposition (\ref{dec}) and comparing the decomposition for $q=0$ and for $q>0$ yields the result.
\end{proof}
\begin{corollary} 
If $Z(\mathbb{R})=\emptyset$, then $T _{q,n}(X^*)=T _{q,n}(X)$ for any $0\leq q\leq n\leq dim(X)$.
\end{corollary}
\begin{remark}
An analog of Corollary \ref{bma} for $T_{q,q}$ was originally proven by Sch\"ulting in \cite{S:reelle}. There separate arguments are needed to give a decomposition algebraically and a decomposition topologically. An advantage that our uniform proof has is that is entirely algebraic, the homology of real points being expressed in terms of homotopy of the group of algebraic cycles $\mcal{R}_0(X)$.
\end{remark}

\begin{remark} 
Using similar techniques one can prove a decomposition analogous to (\ref{dec}) for the  spaces  of real cycles defining dos Santos equivariant Lawson homology groups.
\end{remark}

\begin{corollary} \label{bic}
The groups
$T _{1,n}(X)$ and $K _{1,n}(X)$ are birational invariants for smooth projective real varieties. 
\end{corollary}
\begin{proof}
By \cite[Theorem 0.3.1]{AKMW:birat} every birational map between smooth projective varieties factors as a composition of blow-ups and blow-downs with smooth centers. The result then follows from Corollary \ref{bma}.
\end{proof}

  We close the section with the following computation.
\begin{corollary} \label{comp}
Let $X$ be a rational smooth  projective surface i.e. $X\stackrel{birational}{\simeq}\mathbb{P}^2 _{\mathbb{R}}$. Then the cycle map
$$
cyc_{q,n}:RL_{q}H _n(X)\to H _n(X(\mathbb{R}),\mathbb{Z}/2)
$$ 
is an isomorphism for $q \leq n$. 
\end{corollary}
\begin{proof} 
We have $\mcal{R}_0(X)=\mcal{R}_0(X(\mathbb{R}))$. By Corollary \ref{bic} we have that $T_{1,n}(X)$ and $K_{1,n}(X)$ are birational invariants and we know that $T_{1,n}(\P^{2}_{\R}) = 0 = K_{1,n}(\P^{2}_{\R})$. By Theorem \ref{dge} the group $\pi_{0}\mcal{R} _2(X)$ is a birational invariant.
\end{proof}
\begin{remark}
 For $X$ as in the previous corollary we have that $RL_{q}H_{n}(X) = 0$ for $n<q$ and $n>2$ and  $RL_{0}H_{0}(X)= RL_{0}H_{2}(X) = RL_{1}H_{2}(X) = RL_{2}H_{2}(X) = \Z/2$.
\end{remark}
  
\section{Coniveau spectral sequences}
In this section we show that the coniveau spectral sequence for reduced morphic cohomology 
collapses. We make use of \cite{CTHK:BO} for the Bloch-Ogus theorem identifying the $E_{2}$-term of this spectral sequence. Let $X$ be a smooth, quasi-projective real variety and write $X^{(p)}$ for the set of points $x\in X$ whose closure has codimension $p$. Let $h^*$ be a cohomology theory with supports. Define $h^{i}_{x}(X) = \colim_{ U\subseteq X}h^{i}_{\overline{x}\cap U}(U)$ and $h^{i}(k(x)) = \colim_{ U \subseteq \overline{x}}h^{i}(U)$ (where in both colimits, $U$ ranges over nonempty opens). One may form the Gersten complex
$$
0\rightarrow \bigoplus_{x\in X^{(0)}} h^{n}_{x}(X)\rightarrow \bigoplus_{x\in X^{(1)}}h^{n+1}_{x}(X) \rightarrow\cdots \to \bigoplus_{x\in X^{(p)}} h^{n+p}_{x}(X) \to \cdots .
$$

This complex gives rise to the coniveau spectral sequence 
$$
E_{1}^{p,q} = \bigoplus_{x\in X^{(p)}}h^{p+q}_{x}(X) \Longrightarrow h^{p+q}(X)
$$
The associated filtration is  $N^{p}h^{n}(X) = \cup_{Z}Im(h^{n}_{Z}(X) \to h(X))$, where the union is over closed subvarieties $Z\subseteq X$ of codimension $p$.

\begin{proposition}\cite[Corollary 5.1.11, Theorem 8.5.1]{CTHK:BO}
Let $h^*$ be a cohomology theory with supports on $Sm/\R$ which satisfies Nisnevich excision and is homotopy invariant. Let $\mcal{H}^{q}$ be the Zariski sheafification of the presheaf $U\mapsto h^{q}(U)$. Then the Gersten complex $E_{1}^{\bullet,q}$ is a flasque resolution of $\mcal{H}^{q}$ and the coniveau spectral sequence has the form
$$
E^{p,q} _2=H^{p}_{Zar}(X;\mathcal{H}^q)\Longrightarrow H^{p+q}(X).
$$
For every $q$, the group $H^0(X,\mathcal{H}^q)$ is a birational invariant for smooth proper varieties. 
\end{proposition}

\begin{corollary}\label{sm}
 Let $X$ be a smooth quasi-projective real variety. For each $k$ we have spectral sequences
\begin{equation*}
 E_{1}^{p,q}(k)  = \bigoplus\limits_{x\in X^{(p)}}RL^{k-p}H^{q}(k(x)) \Longrightarrow RL^{k}H^{p+q}(X).
\end{equation*}
The $E_{2}$-terms are  $E_{2}^{p,q}(k)  = H_{Zar}^{p}(X;\mcal{RL}^{k}\mcal{H}^{q})$
and each $H_{Zar}^{0}(X;\mcal{RL}^{k}\mcal{H}^{q})$ is a birational invariant for smooth projective real varieties. Moreover, the $s$-maps induce maps of spectral sequences $\{E_{r}^{p,q}(k)\} \to \{E_{r}^{p,q}(k+1) \}$.
\end{corollary}
\begin{proof}
By Corollary \ref{boprops} reduced morphic cohomology is homotopy invariant and satisfies Nisnevich excision. By Theorem \ref{reducedpurity}  we have for any $x\in X^{(p)}$ the isomorphism $RL^{k}H^{p+q}(X)_{x}=RL^{k-p}H^{q}(k(x))$. Thus the $E_{1}$-page of the coniveau spectral sequence can be rewritten in the displayed form. The $s$-maps are natural transformations and so induce maps of the exact couples defining the coniveau spectral sequence.
\end{proof}

The following result gives us the collapsing of the coniveau spectral sequence. It is  a consequence of the main vanishing theorem of \cite{HV:VT} and the local vanishing of equivariant morphic cohomology and morphic cohomology. 
\begin{theorem}\label{vanishing}
Let $R = \mcal{O}_{T,t_{1},\cdots, t_{n}}$ be the semi-local ring of a smooth, real variety $T$ at the closed points $t_{1},\ldots, t_{n}\in T$. Then 
$$
RL^{k}H^{q}(\spec R)= 0
$$ 
if $q\neq 0$ and any $k>0$. 
\end{theorem}
\begin{proof}
For convenience write $Y=\spec R$. Recall that by definition 
$$
RL^{k}H^{q}(Y)=\colim  RL^{k}H^{q}(U),
$$ 
where the colimit is over all open $U\subseteq T$ such that all $t_{i}\in U$.
Recall also that filtered colimits commute with homotopy groups and preserve exact sequences.

We need to see that $RL^{k}H^{k-s}(Y) = \pi_{s}\mcal{R}^{k}(Y) = 0$ for $s\neq k$. The main vanishing result in \cite[Theorem 6.10]{HV:VT} implies that $\pi_{s}\mcal{R}^{k}(Y) = 0$ for $s> k$.

Consider the homotopy fiber sequences of simplicial abelian groups
\begin{equation*}
 \xymatrix{
\widetilde{\mcal{Z}}^{k}/2(Y_{\C})^{G} \ar[r] & \widetilde{\mcal{Z}}^{k}/2(Y_{\C}) \ar[r]^{N} & \widetilde{\mcal{Z}}^{k}/2(Y_{\C})^{av}  
}
\end{equation*}
and
\begin{equation*}
 \xymatrix{
\widetilde{\mcal{Z}}^{k}/2(Y_{\C})^{av} \ar[r] & \widetilde{\mcal{Z}}^{k}/2(Y_{\C})^{G} \ar[r] & \widetilde{\mcal{R}}^{k}(Y)  .
}
\end{equation*}
Because $\pi_{s}\widetilde{\mcal{Z}}^{k}/2(Y_{\C})^{G} = 0 = \pi_{s}\widetilde{\mcal{Z}}^{k}/2(Y_{\C})$ if $s\leq k-1$ (\cite[Theorem 7.3]{FHW:sst} and \cite[Lemma 3.22]{HV:AHSS}) we see that $\pi_{s}\widetilde{\mcal{Z}}^{k}/2(Y_{\C})^{av} = 0$ if $s\leq k-1$. Using the second homotopy fiber sequence we conclude that $\pi_{s}\widetilde{\mcal{R}}^{k}(Y) = 0$ if $s\leq k-1$. 
\end{proof}

\begin{corollary}\label{dge}
Let $X$ be a smooth quasi-projective real variety. For any $k$, the
 coniveau spectral sequence for reduced morphic cohomology satisfies 
$$
E_{1}^{p,q}(k) = 0
$$ 
for $q\neq 0$. Consequently $E_{2}^{p,0}(k) = E_{\infty}^{p,0}(k)$ and so we have natural isomorphisms
\begin{equation*}
H^{p}_{Zar}(X;\mcal{RL}^{k}H^{0}) = RL^{k}H^{p}(X).
\end{equation*}
In particular $H^0(X,\mathcal{RL}^qH^0)=RL^qH^0(X)=\pi _q(\mathcal{R}^q(X))$ is a birational invariant for smooth projective real varieties.
\end{corollary}

\begin{remark}\label{shvan}
 The case $k= \dim X$ tells us that $\mcal{H}^{i}_{\R}= 0$ for any $i>0$ where $\mcal{H}^{i}_{\R}$ is the Zariski sheafification of the presheaf $U\mapsto H^{i}(U(\R);\Z/2)$. This also follows from \cite[Theorem 19.2]{Sch:ret}. 
\end{remark}

\begin{remark}
Corollary \ref{dge} gives us birational invariants $RL^qH^0(X)$ for $0\leq q\leq d = \dim(X)$.
If  $q= d$ then we have 
$$
RL^dH^0(X) = H^{0}(X(\R);\Z/2) = (\Z/2)^{\oplus s}
$$ 
and therefore $s=s(X) = \#(\textrm{connected components of}\,\,X(\R))$ is a birational invariant of an algebraic nature. This also follows from the main result of \cite{CTP:real} where they show that $H^0(X,\mathcal{H}^n _{et})=H^0(X(\mathbb{R}),\mathbb{Z}/2)$ for any $n\geq dim(X)+1$. Here $\mathcal{H} _{et}$ is the sheaf associated to the presheaf $U\rightarrow H^n _{et}(U, \mu^{\otimes n} _2)$.

At the other extreme if one takes $q=0$, 
$$
RL^0H^0(X)= (\Z/2)^{\oplus r}
$$  
where $r = r(X)= \#(\textrm{geometrically irreducible components of}\,\,X)$ and so $r$ is also a birational invariant. 

\end{remark}

\begin{remark} \label{cmor}
By Corollary \ref{dge} and Corollary \ref{sm} the $s$-maps
$$
RL^qH^n(X)\stackrel{s}{\rightarrow}RL^{q+1}H^n(X)
$$ 
are obtained as the map induced on Zariski sheaf cohomology by the sheafified $s$-maps 
$s:\mcal{RL}^q\mcal{H}^{0}\to\mcal{RL}^{q+1}\mcal{H}^{0}$.
In particular we see that the generalized cycle map 
$$
cyc_{q,n}:RL^{q}H^{n}(X) \to H^{n}(X(\R);\Z/2)
$$ 
is obtained from the sheafified cycle map $\mcal{RL}^{q}\mcal{H}^{0} \to \mcal{H}^{0}_{\R}$.
In the last section we show that this cycle map is naturally related to the Borel-Haeflinger cycle map \cite{BH:cycle}.
\end{remark}

We finish by observing that Poincare duality gives the collapsing of the niveau spectral sequence for reduced Lawson homology (of possibly singular varieties). 

\begin{proposition}\label{niveau} 
Let $X$ be a quasi-projective real variety. Write $X_{(p)}$ for the set of points $x\in X$ whose closure has dimension $p$. The niveau spectral sequence 
$$
E^1 _{p,q}(k)=\oplus _{x\in X_{(p)}}RL_{k}H _{p+q}(k(x))\Longrightarrow RL_{k}H _{p+q}(X)
$$ 
satisfies $E^{1}_{p,q}(k)=0$ for any $q\neq 0$ and therefore $E^{2}_{p,q}(k) = E^{\infty}_{p,q}(k)$.
\end{proposition} 
\begin{proof}
 The niveau spectral sequence is constructed as in \cite{BO:gersten}. Let $x\in X_{(p)}$. Using Corollary \ref{dge} we see that for any $q\neq 0$ 
$$
RL_{k}H _{p+q}(k(x)) = RL^{p-k}H^{-q}(k(x)) = 0 .
$$
\end{proof}

\section{Filtrations in homology} 
Let $X$ be a quasi-projective real variety of dimension $d$. The generalized cycle map  $\phi _{q,n}: RL_{q}H_{n}(X)\rightarrow H _n(X(\mathbb{R}),\mathbb{Z}/2)$ is the composition of $q$ iterations of the $s$-map together with the Dold-Thom isomorphism. Write
$$
RT_qH _n(X)=Im(\phi _{q,n}:RL_qH _n(X) \to H_{n}(X(\R);\Z/2).
$$
This gives us a decreasing filtration of the homology of the space of real points and is called the topological filtration.

Associated to the niveau spectral sequence is the niveau filtration 
$$
N_pRL_{k}H _n(X)=\sum_{\dim V \leq p} Im\left(RL_{k}H_n(V)\rightarrow RL_{k}H_n(X)\right).
$$   
Notice that in the complex case, if $Y$ is a complex variety of dimension $d$ then Weak Lefschetz theorem says that 
$$
N_{n}H_n(Y(\C);\Z)=N _{n+1}H_n(Y(\C);\Z)= \cdots = N_{d}H _n(Y(\C);\Z)).
$$
In particular, in the complex case there are only $n+1$ steps in the filtration, and another $d-n$ are equal to the homology. In the real case, we don't have this theorem and so apriori all one has is a filtration.
$$
N_{0}H _n(X(\R);\Z/2)\subseteq \cdots \subseteq N_{d}H _n(X(\R);\Z/2)=H _n(X(\R);\Z/2).
$$

Teh has formulated the following conjecture which is made in analogy with a conjecture of Friedlander-Mazur \cite[Conjecture p.71]{FM:filt} for complex varieties.
\begin{conjecture}[{\cite[Conjecture 7.9]{Teh:real}}]\label{Tq} 
Let $X$ be a smooth projective real variety. Then  $RT _qH _n(X)\subseteq N_{2n-q}H _n(X(\R);\Z/2)$ and moreover this containment is an equality $RT _qH _n(X) =N _{2n-q}H _n(X(\R);\Z/2)$ for any $0\leq q\leq n$.
\end{conjecture}   
   
From Proposition \ref{niveau} we have the following equality: 
$$
E^{\infty}_{n-q,q}(k)=N _{n-q}RL_{k}H_{n}(X)/N _{n-q-1}RL_{k}H_{n}(X)=0
$$ 
for any $q\neq 0$. This means that for any $k$ we have 

$$
RL_{k}H_{n}(X) = N_dRL_{k}H_{n}(X) = \cdots = N_{n+1}RL_{k}H_{n}(X) = N_{n}RL_{k}H_{n}(X),
$$
and 
$$
0 = N_{-1}RL_{k}H_{n}(X) = N_{0}RL_{k}H_{n}(X) =\cdots = N_{n-1}RL_{k}H_{n}(X).
$$

The first row of equalities contains the groups that appear in Conjecture \ref{Tq}. 
Consequently the first part of the conjecture is obviously true because by the above we have that $N_{j}H_{n}(X(\R);\Z/2) = H_{n}(X(\R);\Z/2)$ for all $j\geq n$.

The second part of the conjecture is false because the $s$-maps are not always be surjective. Using the material from Section \ref{BIE} we give an explicit example of this failure. Recall that we write $T_{q,n}(X) = \coker(cyc_{q,n}:RL_{q}H_{n}(X) \to H_{n}(X(\R);\Z/2))$.

\begin{example}\label{example}
Let $Z\subseteq \P^{3} _\R$ be the smooth irreducible elliptic curve  given by the equation $t^2x+ty^2-x^3=0$, $z=0$. Then $Z(\R)$ is well known to consist of 2 connected components (see for example \cite[Example 3.1.2]{BCR:rag}). Let $X^{*}$ be the blow-up of $\P^{3}$ along $Z$. Then 
$T_{2,2}(X^{*}) = T_{2,2}(\P^{3}_{\R})\oplus T_{1,1}(Z)$ by Corollary \ref{bma}. Because $RL_{1}H_{1}(Z) = \pi_{0}\mcal{R}_{1}(Z) = \Z/2$ and $Z(\R)$ has two components we conclude that $T_{2,2}(X^{*}) = \Z/2$.  
\end{example}

More generally we have the following.
 
\begin{proposition} \label{genex}
For each $N\geq 3$, there is a smooth projective real variety $X$ of dimension $N$ (which is topologically connected) such that $T _{q,q}(X)\neq 0$ for all $2\leq q \leq N-1$.
 \end{proposition}   
 \begin{proof} 
Let $Z\subseteq \P^{N}$ be a smooth irreducible real curve such that $Z(\mathbb{R})$ has at least 2 connected components. Let $s$ denote the number of connected components of $Z(\R)$. Since $Z$ is irreducible we have that $\mcal{R}_1(Z)=\mathbb{Z}/2$. Therefore $T_{1,1}(Z)=(\mathbb{Z}/2)^{s-1}$. Since $Z$ is a curve $T_{i,n}(Z) = 0$ for all other values of $i$ and $n$. We take $X\rightarrow \P^{N}$ to be the blow up of $\P^{N}$ along $Z$. Then $T_{q,q}(X) = T _{1,1}(Z)=\Z/2^{s-1}$ by Corollary \ref{bma} because $2\leq q \leq N-1 = \codim(Z)$. 
\end{proof}

We also have a similar result for the kernel.
\begin{proposition} \label{Kenex}
For each $N\geq 3$, there is a smooth projective real variety $X$, birational to $\P^{N}$, such that $K _{q,q}(X)\neq 0$ for all $2\leq q \leq N-1$. 
 \end{proposition}   
 \begin{proof} 
Let $Z\subseteq \P^{N}$ be a smooth irreducible real curve such that $Z(\mathbb{R})=\emptyset$. Then $K_{1,1}(Z)=\mathbb{Z}/2$ and  $K_{i,n}(Z) = 0$ for all other values of $i$ and $n$. Take $X\rightarrow \P^{N}$ to be the blow up of $\P^{N}$ along $Z$. We have $K_{q,q}(X) = K _{1,1}(Z)=\Z/2$ by Corollary \ref{bma}. 
\end{proof}
As an interesting particular case we have the following which is different than the complex analog.
\begin{corollary} 
There exists  a smooth real variety $X\stackrel{birational}\sim \mathbb{P}^N _\mathbb{R}$ such that the cycle map on divisors $RL^1H^1(X)\rightarrow H^1(X(\R),\Z/2)$ is not injective.
\end{corollary}

\begin{remark}\label{thin}
A $k$-cycle is said to be thin if it is a sum of closed subvarieties $Z\subseteq X$ with $\dim Z(\R) < k$. The kernel of the Borel-Haeflinger cycle map consists entirely of thin cycles by \cite{IS:real} and the composite $CH _q(X)\to RL _qH _q(X)\rightarrow H _q(X(\R);\Z/2)$ agrees with the Borel-Haeflinger cycle map by Theorem \ref{bh}. This means that the proposition above gives examples of nonzero classes which are represented by thin cycles in $RL _qH _q(X)$.
\end{remark}

\section{Cycle Maps} \label{cyclemaps}
Let $X$ be a smooth quasi-projective real variety. We discuss two natural cycle maps from motivic cohomology to the singular cohomology of $X(\R)$. Based on this, we show that Borel-Haeflinger map \cite{BH:cycle} factors through the reduced Lawson homology cycle map. We end the section with a discussion of the maps involved in the Suslin conjecture  from the view of the methods in this section.

 Recall that $G=\Z/2$. If $M$ is a $G$-space we write $H^{i}_{G}(M;\Z/2)$ for the Borel cohomology with $\Z/2$-coefficients. The reduced morphic cohomology of $X$ comes equipped with a natural generalized cycle map to the singular cohomology of real points. Composing this with the canonical map from real morphic cohomology and its isomorphism with motivic cohomology (with $\Z/2$-coefficients) gives us
the cycle map
\begin{equation}\label{themap}
H^{p}_{\mcal{M}}(X;\Z/2(q))  \to RL^{q}H^{p-q}(X) \xrightarrow{cyc_{}}H^{p-q}(X(\R);\Z/2).
\end{equation}
On the other hand the real morphic cohomology maps naturally to the Borel cohomology of the space of complex points. In turn there is a map $H^{n}_{G}(X(\C);\Z/2) \to \oplus H^{n-i}(X(\R);\Z/2)$ obtained by restricting to real points together with the decomposition
$$
H^n _G(X(\mathbb{R}),\mathbb{Z}/2)= H^{n}(X(\R)\times \R P^{\infty};\Z/2) = \bigoplus _{0\leq i\leq n}H^i(X(\mathbb{R}),\mathbb{Z}/2).
$$
Composing with the appropriate projection gives us 
\begin{equation}\label{amap}
H^{p}_{\mcal{M}}(X;\Z/2(q))  \to H^{p}_{G}(X(\C);\Z/2) \xrightarrow{}H^{p-q}(X(\R);\Z/2).
\end{equation}
We show that the cycle maps (\ref{themap}) and (\ref{amap}) agree with each other. Basically these agree because they can be seen as induced by maps of presheaves of cochain complexes and so Theorem \ref{MT} applies. 

Write $(Top)_{an}$ for the category of topological spaces homeomorphic to a finite dimensional CW-complex given the usual topology and $\phi:(Top)_{an}\to (Sm/\R)_{Zar}$ for the map of sites induced by $X\mapsto X(\R)$.
\begin{theorem}\label{thm:cycmaps}
 Let $X$ be a smooth quasi-projective real variety. The cycle maps given by
 (\ref{themap}) and (\ref{amap}) agree. Moreover the intermediate maps in (\ref{amap}) can be chosen so that the following diagram commutes
  $$
\xymatrix{
       H^p _{\mcal{M}}(X,\Z/2(q))\ar[r]^{}\ar[d] & H^p _{et}(X,\mu^{\otimes q} _2) \ar[r]^{\iso}\ar[d] & H^p _G(X(\C),\Z/2)\ar[d]\\     
        RL^pH^{p-q}(X)\ar[r]^{cyc} & H^{p-q}(X(\R),\Z/2)  & \ar[l] H^p _G(X(\R),\Z/2).\\
}
$$
for any $p,q\geq 0$.
\end{theorem}
  \begin{proof}
 Consider the following complexes of Zariski sheaves on $Sm/\R$:
   \begin{align*}
\mathbb{Z}/2(q)(X) & = (z_{equi}(\P_{\R}^{q/q-1},0)(X\times_{\R}\Delta^{\bullet}_{\R})\otimes\Z/2)[-2q] \\
\mathbb{Z}/2(q)^{sst}(X)   & = \sing_{\bullet} (\mcal{Z}^{q}/2(X_{\C})^{G})[-2q] \\
\mathbb{Z}/2(q)^{top}(X)  & = \Hom _{cts}({X(\C)\times\Delta^{\bullet}_{top}},{\mcal{Z}/2_{0}(\A^{q}_{\C})})^{G}[-2q] \\
\mathbb{Z}/2(q)^{Bor}(X) & = \Hom  _{cts}({X(\C)\times EG\times\Delta^{\bullet}_{top}},{\mcal{Z}/2_{0}(\A^{q}_{\C})})^{G}[-2q]\\
\mathbb{Z}/2(q)^{Bor} _\R(X) & = \Hom  _{cts}({X({\R})\times EG\times\Delta^{\bullet}_{top}},{\mcal{Z}/2_{0}(\A^{q}_{\C})})^G[-2q], \\
\mcal{R}(q)(X) &= (\sing_{\bullet} \mcal{Z}^{q}/2(X_{\C})^{G}/\sing_{\bullet}\mcal{Z}^{q}/2(X_{\C})^{av})[-2q]
\end{align*}

These complexes all satsify Nisnevich descent for standard reasons. See e.g. \cite[Section 5]{HV:VT} for the first three. Similarly the complex $\mathbb{Z}/2(q)^{Bor} _\R(X)$ because taking real points of a distinguished Nisnevich square of real varieties gives a  homotopy pushout square of spaces. The last complex satisfies Nisnevich descent by Proposition \ref{boprops}. As a consequence we have

\begin{align*}
\mathbb{H}^{i}_{Zar}(Z; \mathbb{Z}/2(q)) & = H^{i}_{\mcal{M}}(X;\Z/2(q)) \\
 \mathbb{H}_{Zar}^{i}(X;\mathbb{Z}/2(q)^{sst}) & = L^{q}H\R^{i-q,q}(X;\Z/2) \\
 \mathbb{H}_{Zar}^{i}(X;\mathbb{Z}/2(q)^{top}) & = H^{i-q,q}(X(\C);\underline{\Z/2}) \\
 \mathbb{H}^{i}_{Zar}(X;\mathbb{Z}/2(q)^{Bor}) & = H^{i}_{G}(X(\C);\Z/2) \\
\mathbb{H}^{i}_{Zar}(X;\mathbb{Z}/2(q)^{Bor} _\R) & = H^{i}_{G}(X(\R);\Z/2) \\
  \mathbb{H}_{Zar}^{i}(X;\mcal{R}(q)) & = RL^{q}H^{i-q}(X),
\end{align*}
where $L^{q}H\R^{i-q,q}(X;\Z/2)$ denotes Friedlander-Walker's real morphic cohomology \cite{FW:real} and 
$H^{i-q,q}(X(\C);\underline{\Z/2})$ is Bredon cohomology.

The map (\ref{themap}) is induced by the map of complexes
\begin{equation}\label{2ndm}
\mathbb{Z}/2(q)\xrightarrow{1)} \mathcal{R}(q)\xrightarrow{2)}  \phi _*\mathbb{Z}/2[-q].
\end{equation}
The map 1) is given by the ``usual'' cycle map from motivic cohomology to reduced morphic cohomology. It is defined as the composite $\Z/2(q) \to \Z/2(q)^{sst}\to \mcal{R}(q)$. 
The map 2) is obtained by adjunction from the composite
$$
\phi^*(\mathcal{R}(q)) \xrightarrow{\wkeq} Map _{cts}(X(\mathbb{R})\times \Delta^* _{top}, \mcal{R}_0(\mathbb{A}^q _{\mathbb{R}}))[-2q]\xrightarrow{\simeq}\mathbb{Z}/2[-q]
$$
which arises because $\mcal{R} _0(\mathbb{A}^q _{\mathbb{R}})\simeq K(\mathbb{Z}/2,q)$ and any CW complex has an open cover given by  contractibles.

We show that the  map (\ref{amap}) is induced by a composite of maps:
   \begin{equation}\label{1stm}
   \mathbb{Z}/2(q)\xrightarrow{} tr _{\leq 2q}\mathbb{R}\epsilon _*\mathbb{Z}/2 \xrightarrow{3)} \mathbb{Z}/2(q)^{Bor} \xrightarrow{4)}
\mathbb{Z}/2(q)^{Bor} _\R\xrightarrow{5)} \phi _*\mathbb{Z}/2[-q].
   \end{equation}
of Zariski complexes of sheaves in $D^-(Shv _{Zar}(Sm/\mathbb{R}))$ which we now explain. Write $\epsilon :X _{et}\rightarrow X _{Zar}$ for the usual map of sites. The first unlabeled map is the cycle map from motivic cohomology to etale cohomology. The map 3) is obtained  in Proposition \ref{CBE} using Cox's theorem \cite{Cox:real}. The map 4) is obtained by restriction to real points. The map 5) will be obtained from the adjoint of a map $\phi^*(\mathbb{Z}/2(q)^{Bor} _\R)\rightarrow \Z/2[-q]$ as follows. Every CW-complex is locally contractible and so 
$$
\phi^*(\mathbb{Z}/2(q)^{Bor} _\R) \wkeq
\Hom  _{cts}(EG\times\Delta^{\bullet}_{top},{\mcal{Z}/2_{0}(\A^{q}_{\C})})^G[-2q]
$$ 
is a quasi-isomorphism of complexes of Zariski sheaves where the right-hand side is the constant sheaf.   In $D^-_{\Z/2}(Ab)$, any complex is quasi-isomorphic with the complex given by its cohomology. We have that  $H^{p,q}(EG;\underline{\Z/2}) = H^{p+q}(BG;Z/2)$ and therefore
\begin{equation*}
 H^k\Hom  _{cts}(EG\times\Delta^{\bullet}_{top},{\mcal{Z}/2_{0}(\A^{q}_{\C})})^G[-2q] = 
   H^{q-k,q}(EG;\underline{\Z/2})=\Z/2 
\end{equation*}
for $0\leq k \leq 2q$ and is $0$ otherwise. This gives us the map
  $$
\phi^*(\mathbb{Z}/2(q)^{Bor} _\R) \simeq \oplus _{0\leq i\leq 2q} \mathbb{Z}/2[-i]\rightarrow \mathbb{Z}/2[-i],
$$
 which induces the $i^{th}$ projection on cohomology 
 $$
H^n _G(X(\R),\Z/2)=\oplus _{0\leq i\leq n}H^{n-i}(X(\R),\Z/2)\rightarrow H^{n-i}(X(\R),\Z/2).
$$ 
In particular the adjoint of this map for $i=q$ gives us the map 5)
    $$
\mathbb{Z}/2(q)^{Bor} _\R \rightarrow \phi _*\Z/2[-q].
$$
 
Because both map \ref{1stm} and map \ref{2ndm} induce non-trivial maps in cohomology they have to coincide by Theorem \ref{MT}.
   
\end{proof}
In the above proof, we used the following proposition which relies on Cox's theorem identifying the  etale cohomology of a real variety with Borel cohomology.
\begin{proposition} \label{CBE}
There is a quasi-isomorphism  $\rho: tr_{\leq 2q}\R\epsilon _*\Z/2\rightarrow \Z/2(q)^{Bor}$ of complexes of Zariski sheaves. 
\end{proposition}
\begin{proof}  
We show that the canonical map $\Z/2(q)^{Bor}\rightarrow \R\epsilon _*\Z/2$, constructed in \cite[Proposition 5.5]{HV:VT} induces a quasi-isomorphism $\Z/2(q)^{Bor}\rightarrow tr_{\leq 2q}\R\epsilon _*\Z/2$. The map $\rho$ is its inverse in the derived category. Its hypercohomology gives the cycle map  $H^n _G(X,\Z/2)\rightarrow H^n _{et}(X,\Z/2)$, for every $n\geq 0$.

There is a quasi-isomorphism $\mathbb{R}\epsilon _*\mu^{\otimes q} _2 \xrightarrow{\wkeq} \mathbb{R}\epsilon _*\mu^{\otimes q+i} _2$ and a commutative diagram
\begin{equation}\label{comm}
\xymatrix{
{\Z /2}(q)^{Bor}\ar[r]\ar[d]^{\simeq} & tr_{\leq 2q}\mathbb{R}\epsilon _*\mu^{\otimes q} _2 \ar[d] ^{\simeq}\\
tr_{\leq 2q}{\Z /2}(q+i)^{Bor}\ar[r] & tr_{\leq 2q}\mathbb{R}\epsilon _*\mu^{\otimes q+i} _2. 
}
\end{equation}
Take $i=q$. The result follows by showing the bottom map is a quasi-isomorphism. 
In \cite[Section 5]{HV:VT} it is shown that the composite
$$
\Z/2(2q) \to tr_{\leq 2q}\Z/2(2q)^{Bor} \to tr_{\leq 2q}\mathbb{R}\epsilon _*\mu^{\otimes 2q} _2 
$$
is the usual cycle map $\Z/2(2q) \to tr_{\leq 2q}\mathbb{R}\epsilon _*\mu^{\otimes 2q} _2$.  
By Voevodsky's resolution of the Milnor conjecture \cite{Voev:miln} this cycle map is a quasi-isomorphism. This implies that that $H^n _G(X,\Z/2)\rightarrow H^n _{et}(X,\Z/2)$ is a surjective map between finite-dimensional spaces for $n\leq 2q$. By \cite{Cox:real} both vector spaces have the same dimension and so the map is an isomorphism. Therefore $\Z/2(q)^{Bor}\simeq tr_{\leq 2q}\R\epsilon _*\Z/2$.

\end{proof}

Let $a_{Zar}$ denote  Zariski sheafification and define the following sheaves 
\begin{align*}
\mathcal{H}^n_{\mcal{M}}(q) =& a_{Zar}(U\mapsto H^{n}_{\mcal{M}}(U(\mathbb{C})),\mathbb{Z}/2(q))) \\
\mathcal{H}^n _\mathbb{C}(G) =& a_{Zar}(U\mapsto H^{n}_{G}(U(\mathbb{C})),\mathbb{Z}/2)) \\
 \mathcal{H}^n _\mathbb{R}(G) =& a_{Zar}(U\mapsto H^{n}_{G}(U(\mathbb{R}),\mathbb{Z}/2)), \\
 \mathcal{H}^n _{et}(q) =& a_{Zar}(U\mapsto H^{n}_{et}(U,\mu_{2}^{\otimes q})), \\
 \mathcal{H}^{n}_\mathbb{R} =& a_{Zar}(U\mapsto H^{n}(U(\mathbb{R}),\mathbb{Z}/2)). 
\end{align*}

Sheafifying the diagram in Theorem \ref{thm:cycmaps} for $p=q$ gives the commutative diagram 
 $$
\xymatrix{
       \mcal{H}^q _{\mcal{M}}(q)\ar[r]^{\iso}\ar[d] & \mcal{H}^q _{et}(q)\ar[r]^{\iso}\ar[d] & \mcal{H}^q_{\C}(G)\ar[d]\\     
        \mcal{RL}^q\mcal{H}^{0}\ar[r]^{cyc_{q,0}} & \mcal{H}^{0}_{\R}  & \ar[l] \mcal{H}^q_{\R}(G).\\
}
$$

By Corollary \ref{dge} and Corollary \ref{sm} the sheafified cycle map $cyc_{q,0}$ induces the usual cycle map between reduced morphic cohomology and singular cohomology
$$
RL^{q}H^{n}(X)= H^{n}_{Zar}(X;\mcal{RL}^{q}\mcal{H}^{0})  \to  H^{n}_{Zar}(X;\mcal{H}^{0}_{\R})= H^{n}(X(\R);\Z/2).
$$

The map $\mcal{H}^q _{\mcal{M}}(q)\to \mcal{H}^q _{et}(q)$ induces the Bloch-Ogus isomorphism 
$$
CH^{q}(X) = H^{q}_{Zar}(X; \mcal{H}^q _{\mcal{M}}(q)) \iso H^{q}_{Zar}(X; \mcal{H}^q _{et}(q)).
$$

The composite $\mcal{H}^q _{et}(q) \to \mcal{H}^q_{\C}(G)\to \mcal{H}^q_{\R}(G) \to \mcal{H}^{0}_{\R}$  induces a map 
$$
CH^{q}(X) = \mcal{H}^{q}_{Zar}(X;\mcal{H}^{q}_{et}(q)) \to \mcal{H}^{q}_{Zar}(X;\mcal{H}^{0}_{\R}) = H^{q}(X(\R);\Z/2)
$$
 which by \cite[Remark 2.3.5]{CTS:zero} is just the Borel-Haeflinger cycle map sending a closed irreducible $Z\subseteq X$ to the Poincare dual of the fundamental class of $Z$ if $\dim Z(\R) = \dim Z$ and zero otherwise. The above commutative diagram tells us that this agrees with the composite $\mcal{H}^{q}(q)_{et} \iso \mcal{H}^{q}_{\mcal{M}}(q) \to \mcal{LR}^{q}\mcal{H}^{0} \to \mcal{H}^{0}_{\R}$ and so we immediately obtain the following.
  
\begin{theorem}\label{bh}
Let $X$ be a smooth quasi-projective real variety. For any $q\geq 0$, the Borel-Haeflinger cycle map 
factors as the composite  
$$
CH^q(X)/2\rightarrow  RL^qH^q(X)\xrightarrow{cyc_{q,n}}  H^q(X(\mathbb{R}),\mathbb{Z}/2),
$$
where the first map is the natural quotient.
\end{theorem}

Next we compare the $s$-map in reduced morphic cohomology with the operation $(-1)$ in \'etale cohomology. Recall that the operation $(-1):H^{i}_{et}(X;\mu_{2}^{\otimes q}) \to H^{i+1}_{et}(X;\mu_{2}^{\otimes q+1})$ is defined to be multiplication with the class $(-1)$ which is  the image of $-1$ under the boundary map $H^{0}_{et}(X;\mathbb{G}_{m})\to H^{1}_{et}(X;\mu_{2})$ in the Kummer sequence. By naturality this is equal to the pullback of $(-1)\in H^{1}_{et}(\spec(\R);\mu_{2})$ under the structure map $X\to \spec(\R)$. Sheafifying gives the operation on Zariski sheaves $(-1):\mathcal{H}^q _{et}(q)\rightarrow \mathcal{H}^{q+1}_{et}(q+1)$.

\begin{proposition} \label{cel} 
Let $X$ be a smooth quasi-projective real variety of dimension $d$. 
The following square commutes for any $q\geq 0$
$$
\begin{CD}
       H^{i}_{Zar}(X;\mcal{H}^{q}_{et}(q)) @>>> RL^qH^{i}(X)\\
                     @VV \cup (-1) V         @VV\cup sV\\
        H_{Zar}^{i}(X;\mcal{H}_{et}^{ q+1}(q+1)) @>>>  RL^{q+1}H^i(X).
\end{CD}
$$
For any $q\geq d+1$ all maps are isomorphisms. 
\end{proposition}
\begin{proof} 
The $s$-operation is induced by multiplication with  $s\in RL^1H^0(\spec(\R))$, where $s$ is the generator. Sheafifying the $s$-map gives a map of Zariski  sheaves and the composite 
$\mathcal{RL}^q\mcal{H}^0\xrightarrow{} \mathcal{RL}^{q+1}\mcal{H}^0\to\mathcal{H}_{\R}^0 $
 induces the usual $s$-map on sheaf cohomology. 

The class $(-1)\in H^1 _{et}(\spec(\R);\mu _2)$ and $(-1)$ maps to $s$ under the isomorphism $H^{1}_{et}(\spec(\R);\mu_{2})\iso RL^{1}H^{0}(\spec(\R))$ because they both map to the generator of $H^{0}(pt;\Z/2)$. Therefore the above square commutes. When $q\geq d$ then the vertical maps are isomorphisms, both $H^{i}_{Zar}(X;\mcal{H}^{q}_{et}(q)) \to H^{i}(X(\R);\Z/2)$ and  $RL^qH^{i}(X) \to H^{i}(X(\R);\Z/2)$ are isomorphisms.
\end{proof}
\begin{corollary}
Let $X$ be a smooth quasi-projective real variety. We have the isomorphism of rings 
$$
H^* _{et}(X,\mu^{\otimes *} _2)[s^{-1}]\simeq RL^*H^*(X)[s^{-1}]\simeq H^*(X(\R),\Z/2),
$$ 
where $s=(-1)$ under the left-hand isomorphism.
\end{corollary}
\begin{remark} 
The map $\mathcal{H}^q _{et}(q)\rightarrow \mathcal{RL}^q{H}^0$ is not in general an isomorphism of sheaves as we can see for the case of a smooth projective  variety of dimension $dim(X)=q$. In this case the first group surjects with non-trivial kernel in codimension $0,1,2$ (under some mild conditions on $X$) into cohomology of real points  $X(\mathbb{R})$ by \cite{CTS:zero}. On the other hand the latter group is the cohomology of real points by \cite{Teh:real}.
\end{remark}

We now turn our attention to natural transformations from various cycle cohomology theories of interest to singular cohomology. 

\begin{theorem} \label{MT} 
Let $k=\R$ or $\C$ and write $\phi: (Top) _{an}\rightarrow (Sm/k) _{Zar}$ the map of sites that send $X\mapsto X(k)$ where $(Top)$ is the category of spaces homeomorphic to finite dimensional CW complexes equipped with the usual topology. Then we have
\begin{enumerate}  
 \item $\Hom _{D^-((Sm/\R)_{Zar})}(\mathbb{Z}/2(q),\, \phi _*\mathbb{Z}/2[-q])=\mathbb{Z}/2$,
\item $\Hom _{D^-((Sm/\R)_{Zar})}(\mcal{R}(q),\phi _*\Z/2[-q]) =\Z/2$,
\item  $\Hom _{D^-((Sm/\C)_{Zar})}(\mathbb{Z}(q)^{sst},\, \mathbb{R}\phi _*\mathbb{Z}/n)=\mathbb{Z}/n$, for any $n\geq 1$.
\end{enumerate}
\end{theorem}
\begin{proof} 
We have a quasi-isomorphism $\Z/2(q)\wkeq \Z/2(q)^{sst}$ and that
$$
\Hom(\mathbb{Z}/2(q)^{sst},\, \phi _*\mathbb{Z}/2[-q])=\Hom(\phi^*(\mathbb{Z}/2(q)^{sst}),\, \mathbb{Z}/2[-q]).
$$ 
Because every CW complex is locally contractible in $D^-(Top)$ we have
$$
\phi^*(\mathbb{Z}/2(q)^{sst})[2q]\simeq \Hom_{cts}(- \times \Delta^{\bullet} _{top},\, \mcal{Z}/2_0(\mathbb{A}^q _{\C})^{G}) \wkeq \sing_{\bullet} \mcal{Z}/2_0(\mathbb{A}^q _{\C})^{G}.
$$
From \cite[(3.6)]{DS:real} it follows that $\mcal{Z}/2_0(\mathbb{A}^q _{\C})^{G} \wkeq \prod_{i=q}^{2q}K(\Z/2,i)$. 
This yields the result because we then have
$$
\Hom_{D^-(Top)}(\phi^*(\mathbb{Z}/2(q)),\, \mathbb{Z}/2[-q]) = \bigoplus_{i=q}^{2q}H^{q}(K(\Z/2,i);\Z/2) = \Z/2.
$$

In the proof of Theorem \ref{thm:cycmaps} we observed that $\phi^*\mcal{R}(q)\simeq \Z/2[-q]$ and so
\begin{multline*}
\Hom _{D^-((Sm/\R)_{Zar})}(\mcal{R}(q), \phi _*\Z/2[-q]) =\Hom _{D^-((Sm/\R)_{Zar})}(\phi^*\mcal{R}(q),\Z/2[-q]) \\
=\Hom _{D^-((Sm/\R)_{Zar})}(\Z/2[-q],\Z/2[-q]) =\Z/2.
\end{multline*}

The last item follows from the equivalence $\phi^{*}\Z(q)^{sst} \wkeq \Z$. We have
 $$
\Hom _{D^-((Sm/\C)_{Zar})}(\mathbb{Z}^{sst}(q), \mathbb{R}\phi _*\mathbb{Z}/n)= \Hom _{D^-((Sm/\R)_{Zar})}(\mathbb{Z},\mathbb{Z}/n)=\mathbb{Z}/n,$$
 for any $n\geq 1$.

\end{proof}

Because $\mathbb{Z}^{sst}(q) = tr_{\leq q}\mathbb{Z}(q)^{sst}$ we have, according to Theorem \ref{MT},  3), that
$$
\Hom _{D^-(Sm/\C)}(\mathbb{Z}(q)^{sst},\, tr _{\leq n}\mathbb{R}\phi_*\mathbb{Z}) = \Z.
$$
Let $\alpha$ be a generator of this group. Recall that Suslin's conjecture is the statement that $\alpha$ is a quasi-isomorphism for all $q\geq 0$.  We end the section with the following corollary of Voevodsky's resolution of the Beilinson-Lichtenbaum conjectures.
 \begin{theorem}
The map
 $\alpha \otimes \mathbb{Z}/n: \mathbb{Z}/n(q)^{sst} \rightarrow tr _{\leq q}\mathbb{R}\phi _*\mathbb{Z}/n$ is a quasi-isomorphism for any  $n\geq 2$. 
 \end{theorem}
 \begin{proof} 
Theorem \ref{MT} and the quasi-isomorphism $\Z(q)^{sst}\otimes\Z/p\wkeq\Z/p(q)^{sst}$ gives
$$
\Hom _{D^-}(\mathbb{Z}(q)^{sst},\, tr _{\leq q}\mathbb{R}\phi _*\mathbb{Z}/p)=\Hom _{D^-}(\mathbb{Z}/p(q)^{sst},\, tr _{\leq q}\mathbb{R}\phi _*\mathbb{Z}/p)=\mathbb{Z}/p
$$ f
or any prime $p>1$. 
Let $\epsilon :(Sm/\C) _{et}\rightarrow (Sm/\C) _{Zar}$ be the usual map of sites.
 The cycle map $\beta: \mathbb{Z}/p(q) \rightarrow tr _{\leq q}\mathbb{R}\epsilon _*\mathbb{Z}/p$ can be seen as a generator of the group  $\Hom_{D^{-}(Sm/\C)}(\mathbb{Z}/p(q)^{sst}, tr _{\leq n}\mathbb{R}\phi _*\mathbb{Z}/p)$ because we have $\mathbb{Z}/p(q)^{sst}\simeq \mathbb{Z}/p(q)$ and $\R\phi _*\mathbb{Z}/p=\R\epsilon _*\Z/p$ by the classical comparison theorem between \'etale and singular cohomology. This means that there is $0<N<p$ such that $\alpha\otimes \Z/p=N\beta$. Since $N$ is prime to $p$ it follows that $\alpha\otimes \Z/p$ is a quasi-isomorphism.
 
 One easily obtains the result for powers of primes and then products of distinct primes using 5-lemma.
  \end{proof}

\bibliographystyle{amsalpha} 
\bibliography{filtrations}

\providecommand{\bysame}{\leavevmode\hbox to3em{\hrulefill}\thinspace}
\providecommand{\MR}{\relax\ifhmode\unskip\space\fi MR }
\providecommand{\MRhref}[2]{%
  \href{http://www.ams.org/mathscinet-getitem?mr=#1}{#2}
}
\providecommand{\href}[2]{#2}
\begin{thebibliography}{AKMW02}

\bibitem[AKMW02]{AKMW:birat}
Dan Abramovich, Kalle Karu, Kenji Matsuki, and Jaros{\l}aw W{\l}odarczyk,
  \emph{Torification and factorization of birational maps}, J. Amer. Math. Soc.
  \textbf{15} (2002), no.~3, 531--572 (electronic). \MR{1896232 (2003c:14016)}

\bibitem[BCR98]{BCR:rag}
Jacek Bochnak, Michel Coste, and Marie-Fran{\c{c}}oise Roy, \emph{Real
  algebraic geometry}, Ergebnisse der Mathematik und ihrer Grenzgebiete (3)
  [Results in Mathematics and Related Areas (3)], vol.~36, Springer-Verlag,
  Berlin, 1998, Translated from the 1987 French original, Revised by the
  authors. \MR{1659509 (2000a:14067)}

\bibitem[BH61]{BH:cycle}
Armand Borel and Andr{\'e} Haefliger, \emph{La classe d'homologie fondamentale
  d'un espace analytique}, Bull. Soc. Math. France \textbf{89} (1961),
  461--513. \MR{0149503 (26 \#6990)}

\bibitem[BO74]{BO:gersten}
Spencer Bloch and Arthur Ogus, \emph{Gersten's conjecture and the homology of
  schemes}, Ann. Sci. \'Ecole Norm. Sup. (4) \textbf{7} (1974), 181--201
  (1975). \MR{MR0412191 (54 \#318)}

\bibitem[Cox79]{Cox:real}
David~A. Cox, \emph{The \'etale homotopy type of varieties over {${\bf R}$}},
  Proc. Amer. Math. Soc. \textbf{76} (1979), no.~1, 17--22. \MR{MR534381
  (80f:14009)}

\bibitem[CTHK97]{CTHK:BO}
Jean-Louis Colliot-Th{\'e}l{\`e}ne, Raymond~T. Hoobler, and Bruno Kahn,
  \emph{The {B}loch-{O}gus-{G}abber theorem}, Algebraic $K$-theory (Toronto,
  ON, 1996), Fields Inst. Commun., vol.~16, Amer. Math. Soc., Providence, RI,
  1997, pp.~31--94. \MR{MR1466971 (98j:14021)}

\bibitem[CTP90]{CTP:real}
J.-L. Colliot-Th{\'e}l{\`e}ne and R.~Parimala, \emph{Real components of
  algebraic varieties and \'etale cohomology}, Invent. Math. \textbf{101}
  (1990), no.~1, 81--99. \MR{1055712 (91j:14015)}

\bibitem[CTS96]{CTS:zero}
J.-L. Colliot-Th{\'e}l{\`e}ne and C.~Scheiderer, \emph{Zero-cycles and
  cohomology on real algebraic varieties}, Topology \textbf{35} (1996), no.~2,
  533--559. \MR{1380515 (97a:14009)}

\bibitem[dS03]{DS:real}
Pedro~F. dos Santos, \emph{Algebraic cycles on real varieties and {${\mathbb
  Z}/2$}-equivariant homotopy theory}, Proc. London Math. Soc. (3) \textbf{86}
  (2003), no.~2, 513--544. \MR{MR1971161 (2004c:55026)}

\bibitem[FG93]{FG:cyc}
Eric~M. Friedlander and Ofer Gabber, \emph{Cycle spaces and intersection
  theory}, Topological methods in modern mathematics (Stony Brook, NY, 1991),
  Publish or Perish, Houston, TX, 1993, pp.~325--370. \MR{MR1215970
  (94j:14010)}

\bibitem[FHW04]{FHW:sst}
Eric~M. Friedlander, Christian Haesemeyer, and Mark~E. Walker,
  \emph{Techniques, computations, and conjectures for semi-topological
  {$K$}-theory}, Math. Ann. \textbf{330} (2004), no.~4, 759--807.
  \MR{MR2102312}

\bibitem[FM94]{FM:filt}
Eric~M. Friedlander and Barry Mazur, \emph{Filtrations on the homology of
  algebraic varieties}, Mem. Amer. Math. Soc. \textbf{110} (1994), no.~529,
  x+110, With an appendix by Daniel Quillen. \MR{MR1211371 (95a:14023)}

\bibitem[Fri98]{F:algco}
Eric~M. Friedlander, \emph{Algebraic cocycles on normal, quasi-projective
  varieties}, Compositio Math. \textbf{110} (1998), no.~2, 127--162.
  \MR{MR1602068 (2000a:14024)}

\bibitem[FW02]{FW:real}
Eric~M. Friedlander and Mark~E. Walker, \emph{Semi-topological {$K$}-theory of
  real varieties}, Algebra, arithmetic and geometry, Part I, II (Mumbai, 2000),
  Tata Inst. Fund. Res. Stud. Math., vol.~16, Tata Inst. Fund. Res., Bombay,
  2002, pp.~219--326. \MR{MR1940670 (2003h:19005)}

\bibitem[HV09]{HV:VT}
Jeremiah Heller and Mircea Voineagu, \emph{Vanishing theorems for real
  algebraic cycles}, American Journal of Mathematics. To appear. Preprint
  available at http://arxiv.org/abs/0909.0569 (2009).

\bibitem[HV10]{HV:AHSS}
\bysame, \emph{Equivariant semi-topological invariants, atiyah's $kr$-theory,
  and real algebraic cycles}, Transactions of the American Mathematical
  Society. To appear. Preprint available at http://arxiv.org/abs/1008.3685
  (2010).

\bibitem[IS88]{IS:real}
Friedrich Ischebeck and Heinz-Werner Sch{\"u}lting, \emph{Rational and
  homological equivalence for real cycles}, Invent. Math. \textbf{94} (1988),
  no.~2, 307--316. \MR{958834 (89k:14004)}

\bibitem[LF92]{LF:qproj}
Paulo Lima-Filho, \emph{Lawson homology for quasiprojective varieties},
  Compositio Math. \textbf{84} (1992), no.~1, 1--23. \MR{MR1183559 (93j:14007)}

\bibitem[Sch85]{S:reelle}
Heinz-Werner Sch{\"u}lting, \emph{Algebraische und topologische reelle {Z}ykeln
  unter birationalen {T}ransformationen}, Math. Ann. \textbf{272} (1985),
  no.~3, 441--448. \MR{799672 (87g:14018)}

\bibitem[Sch94]{Sch:ret}
Claus Scheiderer, \emph{Real and \'etale cohomology}, Lecture Notes in
  Mathematics, vol. 1588, Springer-Verlag, Berlin, 1994. \MR{1321819
  (96c:14018)}

\bibitem[Teh10]{Teh:real}
Jyh-Haur Teh, \emph{A homology and cohomology theory for real projective
  varieties}, Indiana Univ. Math. J. \textbf{59} (2010), no.~1, 327--384.
  \MR{2666482 (2011g:14018)}

\bibitem[Voe03]{Voev:miln}
Vladimir Voevodsky, \emph{Motivic cohomology with {${\bf Z}/2$}-coefficients},
  Publ. Math. Inst. Hautes \'Etudes Sci. (2003), no.~98, 59--104. \MR{MR2031199
  (2005b:14038b)}

\bibitem[{Voi}]{Voin:2}
M.~{Voineagu}, \emph{{Cylindrical Homomorphisms and Lawson Homology}}, Journal
  of K-theory. To appear. Preprint available at http://arxiv.org/abs/0904.3374.

\end{thebibliography}
\end{document}